\documentclass[a4paper,12pt]{amsart}
\usepackage{amsmath,amssymb}
\usepackage{verbatim}
\usepackage{color}

\newcommand\enn{\mathbb N}
\newcommand\err{\mathbb R}

\newcommand{\eps}{\varepsilon}

\newcommand{\vertiii}[1]{{\left\vert\kern-0.25ex\left\vert\kern-0.25ex\left\vert#1\right\vert\kern-0.25ex\right\vert\kern-0.25ex\right\vert}}

\newtheorem{thm}{Theorem}[section]
\newtheorem{prop}[thm]{Proposition}

\newtheorem{cor}[thm]{Corollary}

\theoremstyle{definition}
\newtheorem{defn}[thm]{Definition}
\newtheorem{example}[thm]{Example}

\newtheorem{claim}[thm]{Claim}

\theoremstyle{remark}
\newtheorem{rem}[thm]{Remark}

\title{Strongly extreme points and approximation properties}

\author[T.~A.~Abrahamsen]{Trond A.~Abrahamsen}
\address{Department of Mathematics, University of
  Agder, Postboks 422, 4604 Kristiansand, Norway.}
\email{trond.a.abrahamsen@uia.no}
\urladdr{http://home.uia.no/trondaa/index.php3}

\author[P. H{\'a}jek]{Petr H{\'a}jek}
\address{Mathematical Institute\\ Czech Academy of
  Science\\ \v{Z}itn{\'a} 25, 115 67 Praha 1, Czech Republic and
  Department of Mathematics \\ Faculty of Electrical Engineering, Czech
  Technical University in Prague\\ Zikova 4, 160 00, Prague, Czech Republic}
\email{hajek@math.cas.cz}

\author[O.~Nygaard]{Olav Nygaard}
\address{Department of Mathematics, University of Agder,
  Postboks 422, 4604 Kristiansand, Norway}
\email{Olav.Nygaard@uia.no}
\urladdr{http://home.hia.no/$~$olavn/}

\author[S.~Troyanski]{Stanimir Troyanski}
\address{Institute of Mathematics and Informatics,
Bulgarian Academy of Science, bl.8,
acad. G. Bonchev str. 1113 Sofia, Bulgaria
and
Departamento de Matem{\'a}ticas, Universidad de Murcia,
Campus de Espinardo, 30100 Espinardo (Murcia), Spain
}
\email{stroya@um.es}

\thanks{The second author was financially supported by GACR 16-073785
  and RVO:  67985840. The fourth
  author was partially supported by MTM2014-54182-P and the
  Bulgarian National Scientific Fund under Grant DFNI-I02/10.}
  
\keywords {denting point, strongly extreme point}
\subjclass[2010]{Primary: 46B20; Secondary: 46B04}

\begin{document}
\begin{abstract}
  We show that if $x$ is a strongly extreme point of a bounded closed convex subset of a Banach space and the identity has a geometrically and topologically good enough local approximation at $x$, then $x$ is already a denting point. It turns out that such an approximation of the identity exists at any strongly extreme point of the unit ball of a Banach space with the unconditional compact approximation property. We also prove that every Banach space with a Schauder basis can be equivalently renormed to satisfy the sufficient conditions mentioned. In contrast to the above results we also construct a non-symmetric norm on $c_0$ for which all points on the unit sphere are strongly extreme, but none of these points are denting.  
\end{abstract}

\maketitle
\section{Introduction}\label{sec:intro}
Let $X$ be a (real) Banach space and denote by $B_X$ its unit ball, $S_X$ its
unit sphere, and $X^*$ its topological dual. Let $A$ be a non-empty
set in $X$. By a slice of $A$ we mean a subset of $A$ of the form
\begin{align*}
  S(A, x^*,\eps):=\{x \in A: x^*(x) > M - \eps\}
\end{align*}
where $\eps >0$, $x^* \in X^*$ with $x^* \not=0$, and $M = \sup_{x
  \in A} x^*(x)$. We will simply write $S(x^*,\eps)$ for a slice of a
set when it is clear from the setting what set we are considering slices of. 

\begin{defn} \label{defn:corner-points} Let $B$ be a non-empty 
  bounded closed convex set in a Banach space $X$ and
  let $x \in B$. Then $x$ 
  \begin{enumerate}
     \item [a)] is an extreme point of $B$ if for any $y, z$ in $B$ we
       have \[x =  \frac{y+z}{2} \Rightarrow y = z = x.\] 
     \item [b)] is a strongly extreme point of $B$ if for any
       sequences $(y_n)_{n=1}^\infty,  (z_n)_{n=1}^\infty$ in $B$ we have \[
      \lim_n \|x - \frac{y_n + z_n}{2}\| = 0 \Rightarrow \lim_n \|y_n -
      z_n\| = 0.\] When B is the unit ball, the above condition can be replaced by
      \[\lim_n \|x \pm x_n \| = 1 \Rightarrow  \lim_n \| x_n \| = 0.\] In
      this case we say that the norm is midpoint locally uniformly rotund (MLUR) at $x$.
     \item [c)] is a point of continuity for the map $\Phi:B\to X$ if $\Phi$ is weak to norm continuous at $x$. When $\Phi$ is the identity mapping we just say that $x$ is a point of continuity (PC). 
     \item [d)] is a denting point of $B$ if for every $\eps > 0$ and
       $\delta >0$ there exists a slice $S(x^*,\delta)$ of $B$ with diameter less than $\eps.$
     \item [e)] is a locally uniformly rotund (LUR) point of $B_X$ if for any sequence
       $(x_n)_{n=1}^\infty$ we
       have \[\lim_n \|x + x_n\| = 2\lim_n \|x_n\|=2\|x\|=2 \Rightarrow \lim_n \|x - x_n\| = 0.\]
    \end{enumerate}
\end{defn}

It is well known that LUR points are denting points and that denting points
are strongly extreme points \cite{MR661446}. Trivially strongly
extreme points are extreme points.

The importance of denting points became clear in the sixties when the Radon-Nikod\'{y}m Property (RNP) got its geometric description. In particular, it became clear that extreme points in many cases are already denting as every bounded closed convex set in a space with the  
RNP has at least one denting point. The ``extra'' an extreme point needs to become a denting points is precisely described in the following

\begin{thm}\cite{LLT}\label{thm:LLT} Let $x$ be an extreme point of continuity of a bounded closed convex set $C$ in $X$. Then $x$ is a denting point of $C$.
\end{thm}

It is well known that all points of the unit sphere of $\ell_1$ are points of continuity for the unit ball $B_{\ell_1}$. So from Theorem \ref{thm:LLT} we get that every extreme point of the unit ball of any subspace of $\ell_1$ automatically gets the ``extra'' to become denting.

However, despite the theoretical elegance of Theorem~\ref{thm:LLT}, it is not always easy to to check whether the identity mapping is weak to norm continuous at a certain point of a bounded closed convex set. For this reason it is natural to look for geometrical conditions which ensure weak to norm continuity of the identity operator at $x$ when we approximate it strongly by maps that are weak to norm continuous at $x$.

One such idea could be to assume that $x$ is strongly extreme (not just extreme as in Theorem~\ref{thm:LLT}) and that the identity map is approximated strongly by finite rank operators. But this is not enough to give the extreme point the ``extra'' needed to be denting: Consider $x=x(t)\equiv 1\in B_{C[0,1]}$. Then $x$ is strongly extreme in $B_{C[0,1]}$, but the identity map $I:B_{C[0,1]}\to B_{C[0,1]}$ is not weak to norm continuous there (see the next paragraph), and $\lim_n\|P_n x-x\|= 0$, where $(P_n)$ are the projections corresponding to the Schauder basis in $C[0,1]$. Clearly $P_n$ is weak to norm continuous at any point of $B_{C[0,1]}$ (as any compact operator is).

Actually $C[0,1]$ belongs to the class LD2 of Banach spaces where all slices of the unit ball have diameter 2. Naturally, in such spaces no point of the unit sphere can be a PC point. See e.g. the references in \cite{MR3499106} for more information about the class LD2.  

Assuming $x$ is strongly extreme, we need to make stronger assumptions of the approximating sequence of the identity. One such condition which we impose is related to the behaviour of the approximating mappings close to the point $x$. In particular we obtain as a corollary that in
Banach spaces with the unconditional compact approximation property
(UKAP) (see Definition \ref{def:ucap}), every strongly extreme point in the unit ball is PC and therefore denting. In particular, we have that this conclusion holds for Banach spaces with an unconditional basis with unconditionally basis constant 1. Further we show that every Banach space with a Schauder basis can be renormed to satisfy the conditions of Theorem \ref{thm:se=dent}. 

Nevertheless we construct, in Section~\ref{sec:cnullgeo}, a closed convex body in $c_0$ where all boundary points are strongly extreme, but none of them is denting. This body is not symmetric; we refer to its gauge as a non-symmetric equivalent norm on $c_0$. Thus, we construct a non-symmetric equivalent norm on $c_0$ for which all points
on the unit sphere are strongly extreme points, but none of these
points is denting. In fact every slice of the unit ball of $c_0$ with
this non-symmetric norm has diameter at least $1/\sqrt{2}.$ 

In Section \ref{sec:eqnorms} we investigate when equivalent norms conserve strongly extreme and denting points of the corresponding unit balls.

The notation and conventions we use are standard and follow
\cite{MR1863688}. When considered necessary, notation and concepts are
explained as the text proceeds.

\section{Weak to norm continuity of the identity map}\label{sec:se=dent}

Our most general result on how to force a strongly extreme point $x$ to be denting in terms of approximating the identity map $I:X\to X$ at $x$ is the following
 
\begin{thm}\label{thm:se=dent}
  Let $x$ be a strongly extreme point of a non-empty bounded closed convex set $C$ in a Banach space $X$. Let $x$ be a point of continuity for a sequence $\Phi_n:C\to X, n=1,2,\ldots $ of maps such that
  \begin{align}
    \label{eq:36}
    \lim_{n}\|\Phi_n x - x\| = 0
  \end{align}
  and 
  \begin{align}
    \label{eq:26}
    \lim_n\lim_{\eps \to 0^+} f_n(\eps) = 0,
  \end{align}
where
\begin{align*}
  f_n(\eps) = \sup\{\mbox{dist}((1+\lambda)\Phi_n y -\lambda y, C): y \in C,
    \|\Phi_nx-\Phi_ny\| \le \eps\},
\end{align*}  
  for some $\lambda \in (0,1].$
   Then $x$ is a denting point of $C$.
\end{thm}

The proof follows from Theorem~\ref{thm:LLT} and the next proposition which is an interplay between weak and norm topology. With $B(x,\rho)$ we denote the ball with center at $x$ and radius $\rho$.

\begin{prop}\label{prop:PC} Let $x$ be a strongly extreme point of a convex set $C$ of a normed space $X$ and let $0<\lambda\leq 1$. Assume that for every $\eta>0$ there exist a weak neighbourhood $W$ of $x$ and a map $\Phi:W\cap C\to X$ such that
\begin{align} 
\label{eq:sk3}
\Phi(W\cap C)\subset B(x,\eta)
\end{align}
and
\begin{align} 
\label{eq:sk4}
 \sup_{w\in W\cap C}\mbox{dist}((1+\lambda)\Phi w -\lambda w,C)<\eta.
\end{align} 
Then $x$ is PC.
\end{prop}

\begin{proof}
Since $x$ is strongly extreme point, for every $\varepsilon>0$ we can find $\delta>0$ such that
\begin{align} 
\label{eq:sk5}
\|x-\frac{u+v}{2}\|<\delta, \quad u,v\in C \Rightarrow \|u-v\|<\lambda\varepsilon.
\end{align}
Set $\eta=\min\{\delta,\lambda\varepsilon\}/2$. There is a weak neighbourhood $W$ of $x$ and a map $\Phi$ satisfying (\ref{eq:sk3}) and (\ref{eq:sk4}). Set $\Psi=I-\Phi$ and pick an arbitrary $w\in W\cap C$.    Put $y^+ = (1 - \lambda)x  + \lambda w.$ Since $x, w \in C$
   we get $y^+ \in C$ by convexity. Since 
  \begin{align*}
    \Phi w + \lambda \Psi w - y^+ = (1 - \lambda)(\Phi w -x),
  \end{align*}
  we have from (\ref{eq:sk3})
\begin{align}
\label{eq:sk6}
    \|\Phi w + \lambda \Psi w  - y^+\| \le (1 - \lambda) \eta <  \eta.
\end{align}
Having in mind (\ref{eq:sk4}) we can find $y^-\in C$ such that 
\begin{align}
\label{eq:sk7}
\|(\Phi w-\lambda \Psi w)-y^-\|<\eta.
\end{align}
   This and (\ref{eq:sk6}) imply
   \begin{align*}
     \|x - \frac{y^+ + y^-}{2}\| & \le \|x - \Phi w\| +
                                       \frac{1}{2}\|(\Phi w + \lambda
                                       \Psi w - y^+)+(\Phi w
                                       -\lambda \Psi w -
                                       y^-)\|\\ &< \|x -
                                                        \Phi w\| +
                                                        \eta \leq 2\eta.\\
   \end{align*}
From (\ref{eq:sk5}) we get
\[\|y^+-y^-\|<\lambda\varepsilon.\]
On the other hand, using again (\ref{eq:sk6}) and (\ref{eq:sk7}), we get
 \begin{align*}
      \|y^+ - y^-\| & = \|y^+ - (\Phi w + \lambda \Psi w) -
                          y^- + (\Phi w - \lambda \Psi w) +
                          2\lambda \Psi w\| \\ & > 2\lambda \|\Psi w\| - 2\eta.
    \end{align*}
Hence
\[2\lambda \|\Psi w\|<\|y^+-y^-\|+2\eta<\lambda\varepsilon+2\lambda\varepsilon=3\lambda\varepsilon.\]
This and (\ref{eq:sk3}) imply
\[\|w-x\|\leq\|\Phi w-x\|+\|\Psi w\|<2\varepsilon.\]
Since $w$ is an arbitrary element of $W\cap C$ we get that $W\cap C\subset B(x,2\varepsilon)$.
\end{proof}

\begin{rem}\label{rem:afterPC} If $x$ is PC for $C$ we get that $x$ satisfies the hypotheses of Proposition~\ref{prop:PC} just taking $\Phi=I, \lambda\in (0,1]$.
\end{rem}

\begin{proof}[Proof of Theorem~\ref{thm:se=dent}]
Let $\{\varepsilon_n\}$ be a sequence of positive numbers tending to 0. Since $\Phi_n:C\to X, n=1,2,\ldots$ is weak to norm continuous at $x$ there is weak neighbourhood $V_n$ of $x$ such that 
\[\Phi_n(V_n\cap C)\subset B(x,\varepsilon_n),\:\:n=1,2,\ldots\]
Thus the conditions of Theorem~\ref{thm:se=dent} imply that for every $\eta>0$ we can find $n=n(\eta)$ such that (\ref{eq:sk3}) and (\ref{eq:sk4}) hold for $W=V_n$
and $\Phi=\Phi_n$.
\end{proof}

Recall that every linear compact operator is weak to norm continuous on bounded sets. This together with Theorem~\ref{thm:se=dent} gives

\begin{cor}\label{rem:se=dent}
  Let $(X, \|\cdot\|)$ be a Banach space and let $\vertiii
  \cdot$ be an equivalent (not necessarily symmetric)
  norm on $X$ with corresponding unit ball $C.$ Let $x$ be a strongly extreme point of $C$. Let $T_n:X\to X$ be linear compact operators such that 
  \begin{align}
    \label{eq:29a}
    \lim_n \|T_nx-x\|=0,
  \end{align}
   \begin{align}
    \label{eq:29}
    \lim_n\lim_{\eps\to 0^+} \sup\{\vertiii{(1+\lambda)T_ny - \lambda y}: \vertiii{y}
    \le 1, \|T_n(x - y)\| \le \eps\} = 1,
  \end{align}
for some $\lambda\in (0,1]$. In particular the above is satisfied if
 \begin{align}
    \label{eq:29b}
    \lim_n \vertiii{(1+\lambda)T_n-\lambda I}=1.
  \end{align}
Then $x$ is a denting point of $C$.
\end{cor}
\begin{proof}
It is enough to prove that $(\ref{eq:29})$ implies $(\ref{eq:26})$. 
Indeed, since there exists $k >0$ such that
  $\|\cdot \| \le k\vertiii{\cdot},$ then for every $u \in X \setminus
  C$ we
   have
   \begin{align*}
     \mbox{dist}(u, C) &= \inf \{\|u-v\|: v \in C\} \le k
     \inf\{\vertiii{u-v}: v \in C\}\\ & \le k\vertiii{u -
                                        \frac{u}{\vertiii{u}} } =
                                        k(\vertiii{u} - 1).
\end{align*}
\end{proof}

\begin{rem}\label{rem:lim-se=dent}
  The functions $f_n$ defined in Theorem~\ref{thm:se=dent} can be discontinuous at
  $0$. Indeed, $X$ be a Banach space, $e \in B_X$, and
  $e^* \in S_{X^*}$ be such that $e^*(e) = \|e\| = \|e^*\| =1.$ Define a
  (norm one) projection $P$ on $X$ by $Px = e^*(x)e$ and put 
 \begin{align*}
    f(\eps) = \sup\{ \|Py - Ry\|: \|y\| \le 1, \|P(e - y)\| \le \eps\},
  \end{align*}
  where $R=I-P$.
  Now, if the norm $\|\cdot\|$ on $X$ is either strictly convex or Gateaux
  differentiable at $e^*,$ then $f$ is discontinuous at
  $0$. Indeed, let $\eps=0, \|y\| \le 1$, and $P(e - y) = 0$. We get $e^*(e - y)e
  = 0$. Hence $e^*(y) =  e^*(e) = 1$.
  By the strict convexity of the norm or the Gateaux differentiability of the norm
  at $e^*,$ we have $y =
  e$. This implies $Ry = Re = 0,$ so $f(0) = 1$. In order to prove
  that $f$ is discontinuous at $0$ we simply apply Corollary~\ref{rem:se=dent} with
  $T_n = P$. Since $e$ is strongly extreme, but not denting, we get
  $\lim_{\eps \to 0^+} f(\eps) > 1.$
\end{rem}
  
 Remark \ref{rem:lim-se=dent} shows that one cannot replace the limit
 condition
 (\ref{eq:29}) of Corollary~\ref{rem:se=dent}, by 
 \[\lim_n\sup\{\vertiii{(1+\lambda)T_ny-\lambda y}:\vertiii{y}\leq 1, T_nx=T_ny\}=1. \]

The conditions in Corollary~\ref{rem:se=dent} are essential. Let us 
illustrate this by examples.

\begin{example}\label{ex:app-ess}
  Consider the space $c$ of
  convergent sequences endowed with its natural norm. Let $e = (1, 1,
  \ldots) \in S_c$ and $P_n$ the projection
  on $c$ which projects vectors onto their $n$ first
  coordinates. Clearly $e$ is a strongly extreme point of $B_c$
  which is not denting. Moreover, it is evident that the condition
  $\lim_n\|P_ne - e\| = 0$ fails and that the condition (\ref{eq:29})
  (moreover (\ref{eq:29a}))
  holds for $\lambda = 1$ (and thus for all $ \lambda \in (0,1]).$
  It follows that the approximation condition in Corollary~\ref{rem:se=dent} is essential.
\end{example}

\begin{example}
  Consider again $c$ endowed with its natural norm.  Let $e \in c$ be is as in Example
   \ref{ex:app-ess}. Define a projection $P$ on $c$ by $Px =
   \lim_n x(n)e$ and put $P_n = P$ for all $n$. By construction $P_ne =
   e.$ For $z = (0, 1, 1, \ldots)$ we have $P_n z = P z  = e.$
  Now, for any $\lambda \in (0,1]$ we have
   \begin{align*}
     \|(1+\lambda)P_nz - \lambda z\| &= \|(1 + \lambda)e - \lambda z\|\\
        & = 1 + \lambda.                                  
   \end{align*}
  Thus
  \begin{align*}
   \lim_n\lim_{\eps \to 0^+}\sup\{\|(1+\lambda)P_ny - \lambda y\|: y \in
    B_X, \|P_n(e - y)\| < \eps\|\} \ge 1+ \lambda. 
  \end{align*} 
 It follows that the
 condition (\ref{eq:26}) in Theorem \ref{thm:se=dent} is essential. 
\end{example}

We now present our results in terms of an approximation property introduced and studied by Godefroy, Kalton and Saphar. 

\begin{defn}\cite{GKS}\label{def:ucap} A Banach space $X$ is said to have the
\emph{unconditional compact approximation property (UKAP)} if there
exists a sequence $(T_n)$ of linear compact operators on $X$ such that $\lim_n \|T_nx - x\| = 0$ for every $x \in X$ and $\lim_n\|I - 2T_n\| = 1.$
\end{defn}

Clearly Banach spaces $X$ with the UKAP satisfy condition (\ref{eq:29b}) for $\lambda = 1$. Clearly also Banach spaces with an unconditional basis with basis
constant $1$ have the UKAP (simply put $T_n = P_n$ the projection onto
the $n$ first vectors of the basis). Thus we immediately have the
following corollary.

\begin{cor}
  If $X$ has the UKAP, in particular if $X$ has an unconditional basis
  with unconditional basis constant $1,$ then all strongly extreme points in $B_X$ are
  denting points. 
\end{cor}

Let us mention that the global condition (\ref{eq:29b}) is much stronger than the local condition
(\ref{eq:29}), even in the case when it holds for all $x$ in $S_X$. This
will be clear from the discussion below and in particular from Example
\ref{ex:gc=/=lc} which shows that the condition (\ref{eq:29b}) is
strictly stronger than (\ref{eq:29}). For that example we will use the
following result.

\begin{prop}\label{prop:lur-cond9}
  Let $X$ be a Banach space and $x$ a locally uniformly rotund
  (LUR) point in $S_X$. Let $(T_n)$ be a sequence of linear bounded operators on $X$, with $\lim_n \|T_n\| =
  1,$ and which satisfies condition (\ref{eq:29a}) in
  Corollary~\ref{rem:se=dent} for $x \in S_X$. Then condition (\ref{eq:29}) holds for $\lambda =1$ (and thus for all $\lambda
  \in (0, 1]$). 
\end{prop}

\begin{proof}
  Pick an arbitrary sequence $(\eps_n)$ with $\eps_n > 0$ and $\lim_n
  \eps_n = 0.$ First we show that
  \begin{align}
    \label{eq:31}
    \lim_n \mbox{diam} D_n = 0,
  \end{align}
  where $D_n = \{y \in B_X: \|T_n(x-y)\| < \eps_n\}.$ To this end note
  that it suffices to show that
  \begin{align}
    \label{eq:32}
    y_n \in B_X, \hskip 2mm \lim_n \|T_n(x-y_n)\| = 0 \Rightarrow
    \lim_n \|x - y_n\| = 0. 
  \end{align}
   Indeed,
   \begin{align*}
     \|T_n\|\,\|x + y_n\|& \ge \|T_n(x + y_n)\|\\&= \|2T_nx + T_n(y_n-x)\|\\&\ge
     2\|T_nx\| - \|T_n(y_n - x)\|.
   \end{align*}
   Hence $\liminf_n \|x + y_n\| \ge 2.$ Since $\|y_n\| \le 1$ we get
   $\lim_n\|x + y_n\| = 2.$ Since $x$ is a LUR point, we get that
   $(\ref{eq:32})$ holds, and thus (\ref{eq:31}) holds. In order to
   prove (\ref{eq:29}) for $\lambda = 1,$ it is enough to show that
  \begin{align*}
    \lim_n d_n =1,  
  \end{align*}
  where $d_n = \sup\{\|T_nx - R_ny\|: y \in D_n\}, R_n=I-T_n.$ Since $x\in D_n$, we have $d_n\geq\|T_nx-R_nx\|$. So we get from (\ref{eq:29a}) 
  that $\liminf d_n \ge 1.$ Now, pick arbitrary $y \in D_n.$ Then we have
  \begin{align*}
    \|T_nx - R_ny\| & \le \|T_nx\| + \|R_ny\|\\
     & \le \|T_nx\| + \|R_nx\| + \|R_n(y-x)\|\\
     & \le \|T_nx\| + \|R_nx\| + \|R_n\|\|(y-x)\|\\
     & \le \|T_nx\| + \|R_nx\| + (\|T_n\| + 1)\mbox{diam} D_n.                    
  \end{align*}
   Hence $\limsup d_n \le 1.$
\end{proof}

\begin{prop}\label{prop:lur}
  Let $(T_n)$ be a bounded sequence of linear compact operators on $X$,  $R_n=I-T_n$, and $(f_n) \subset S_{X^*}$
  a separating family for $X$. Then the norm
  \begin{align*}
    \vertiii{u} = \bigg( \sum_{n=1}^\infty 2^{-n}(\|R_nu\|^2 + f^2_n(u) )\bigg)^{\frac{1}{2}}
  \end{align*}
   is LUR at $x \in X$ provided $\lim_n\|R_nx\| = 0.$ Moreover, if
   the operators $(T_n)$ commute and $\lim_n\|T_n\| = 1,$ then
   $\lim_n\vertiii{T_n} = 1.$ 
\end{prop}

\begin{proof}
  Pick a sequence $(x_k) \subset X$ with $\lim_k\vertiii{x_k +
    x}=\vertiii{x} = \vertiii{x_k}.$ By convex arguments
  (\cite[Fact~2.3 p.~45]{MR1211634}) we have
  \begin{align}
    \label{eq:33}
    \lim_k\|R_n x_k\| &= \|R_n x\|, \hskip 2mm n = 1, 2, \ldots
    \end{align}
    \begin{align}
      \label{eq:34}
       \lim_k f_n(x_k) &= f_n(x), \hskip 2mm n=1, 2, \ldots.   
    \end{align}
   
  First we show that $(x_k)$ is norm compact. Given $\eps > 0$ we can
  find $n$ with $\|R_nx\| < \eps.$ Using (\ref{eq:33}) we can find
  $k_\eps$ such that $\|R_nx_k\| < \eps$ for $k < k_\eps.$ The set $K
  = \{x_1, x_2, \ldots, x_{k_\eps}\} \cup \|x\|T_n(B_X)$ is norm
  compact. We show that $K$ is an $\eps$-net for $(x_k).$ Indeed, pick
  $x_k, k > k_\eps.$ Then $\|x_k - T_nx_k\| = \|R_nx_k\| < \eps$ and
  $T_nx_k \in \|x\| T_n(B_X).$ So $(x_k)$ is norm compact. Since
  $(f_n)$ is total, we get from (\ref{eq:34}) that $\lim_k\|x_k -x\| =
  0.$ Thus the norm $\vertiii{\cdot}$ is LUR at the point $x \in X$.

 Now, let us prove the moreover part. As $(T_n)$
 commute, we have
 \begin{align*}
   \vertiii{T_mu}^2 &= \sum_{n=1}^\infty 2^{-n} (\|R_nT_mu\|^2 +
                      f_n^2(T_mu))\\
    & = \sum_{n=1}^\infty 2^{-n} (\|T_mR_nu\|^2 + (T^*_mf_n(u))^2) \\
    & \le \sum_{n=1}^\infty 2^{-n} (\|T_m\|^2\|R_nu\|^2 +
      \|T^*_m\|^2f_n^2(u)) \\
    & = \|T_m\|^2 \vertiii{u}^2.
 \end{align*}
  Hence $\vertiii{T_m}^2 \le \|T_m\|^2$ for all $m = 1, 2, \ldots,$ so
  $\limsup_m \vertiii{T_m} \le \limsup_m\|T_m\|.$ Since $\lim_m\|T_mx
  - x\| = 0,$ we get $\liminf_m \vertiii{T_m} \ge 1,$ and so
  $\lim_m\vertiii{T_m} =1$ provided $\lim_m \|T_m\| = 1.$
\end{proof}

It is now easy to give the announced example which shows that condition
(\ref{eq:29b}) can fail as condition
(\ref{eq:29}) holds for every $x$ in $S_X$.

\begin{example}\label{ex:gc=/=lc}
  Consider $c_0$ endowed with the norm
  $\|\cdot\|$ defined by $\|x\| = \sup_{i,j \ge 1} ( x(i) -
  x(j) )$ where $x=(x(k)) \in c_0.$ 
  Clearly $\|\cdot\|$ is equivalent to the canonical norm on $c_0$. 
  Let $P_n$ be the projection onto the $n$ first vectors in the
  canonical basis $(e_k)$ of $c_0$ and let $\vertiii
  \cdot$ be the norm on $c_0$ given in Proposition
  \ref{prop:lur} where $f_n = 0$ for every $n$. Then $(c_0,
  \vertiii \cdot)$ fulfils the conditions of
  Proposition \ref{prop:lur-cond9} and thus satisfies condition
  (\ref{eq:29}) for every $x$ in $S_{c_0}$. Nevertheless we have
  $\vertiii {P_k - \lambda R_k} >
  1$ for any $\lambda \in (0,1]$, so condition
(\ref{eq:29b}) fails. For the latter,  just consider $(P_k - \lambda
  R_k)(\sum^{k+1}_{i=1}e_i).$ 
\end{example}

From the two preceding propositions we also get

\begin{cor}
  Let $X$ be a Banach space with a Schauder basis. Then there exists
  an equivalent norm $\|\cdot\|$ on $X$ for which the sequence of projections
  $P_n$ onto the first $n$ vectors of the basis, satisfy (\ref{eq:29}) for $\lambda = 1$.
\end{cor}

On the other hand we have 

\begin{prop}
  There exists an equivalent norm $\|\cdot\|$ on $C[0,1]$ such that (\ref{eq:29}) does not hold for any $\lambda > 0$ and any
  sequence $(T_n)$ of compact linear operators on $X$when $x \in C[0,1]$ with
  $\|x\| = 1$ and $\lim_n \|T_nx -x\|_{\infty} = 0.$
\end{prop}

\begin{proof}
  The norm on $C[0,1]$ constructed in \cite[Theorem~2.4]{MR3499106}
  is midpoint locally uniformly rotund and has the diameter two
  property, i.e. all non-empty relatively weakly open subsets of the
  unit ball have diameter $2$. In particular, in this norm all points
  on the unit sphere are strongly extreme, but none are denting. Thus
  the conclusion follows from Theorem \ref{thm:se=dent}.
\end{proof}

\section{An MLUR non-symmetric norm on $c_0$ without denting points on the unit sphere}\label{sec:cnullgeo}

We will now construct a non-symmetric equivalent norm on $c_0$ for
which all points on the unit sphere are strongly extreme, but none of
these points are denting. 

\begin{prop}
  There exists in $c_0$ an equivalent (non-symmetric) MLUR norm
  $\vertiii{\cdot}$ with no denting points. Moreover, every open slice
  of the unit ball corresponding to $\vertiii{\cdot}$ has diameter
  $\ge 1/\sqrt{2}.$
\end{prop}

\begin{proof}
  With $J$ we denote the set of all finite strictly increasing
  sequences $(j_k)_ {k \ge 0}$ of natural numbers. For $x = (x(k))_{k
    \ge 1} \in \ell_\infty$ and $j \in J$ we put
  \begin{align*}
    Q(x,j) &= |x(j_0)| + \sum_{k \ge 1} 2^{-k} x^+(j_k), \hskip 2mm x^+ = \max\{x,0\}\\
    q(x) &= (\sum_{k \ge 1} 2^{-k} x^2(k))^{1/2}, \hskip 2mm
    \|x\| = \sup\{Q(x,j): j \in J\},\\ 
    \vertiii{x} &= (\|x\|^2 + q^2(x))^{1/2}.
   \end{align*}
  Clearly
  \begin{align}
    \label{eq:41}
    \|x\|_{\infty} \le \|x\| \le \vertiii{x} \le (\|x\|^2 +
    \|x\|_\infty^2)^{1/2} \le \sqrt{2} \|x\| \le 3\|x\|_\infty.
  \end{align}
   
  \begin{claim}\label{claim:1}
    For every $x \in c_0$ and every $\eps > 0,$ there exists $m = m(x,
    \delta) \in \enn$ and $\delta = \delta(x,\eps) > 0$ such that
    \begin{align}
      \label{eq:42}
      \max\{\|x+y\|, \|x-y\|\} \ge \|x\| + \delta
    \end{align}
    whenever
    \begin{align}
      \label{eq:43}
      y \in \ell_\infty, \|R_my\|_\infty > \eps,
    \end{align}
    where $P_m$ is the projection onto the first $m$ vectors of the
    canonical basis of $c_0$ and $R_m = I - P_m.$
  \end{claim}
  \begin{proof}
    Pick $m$ such that
    \begin{align}
      \label{eq:44}
      \|R_m x\|_ \infty < \eps/8,
    \end{align}
    and put $\delta = \eps/2^{m+3}.$ Using the definition of
    $\|\cdot\|$ we can find $j = (j_i)_{k=0}^p$ such that
    \begin{align}
      \label{eq:45}
      \|x\| - \delta/2 < Q(x,j).
    \end{align}
    We may assume that $p \ge m.$ Choose $i \ge 1$ such that
    \begin{align}
      \label{eq:46}
      j_{i-1} \le m < j_i 
    \end{align}
    and put $j^1 = (j_k)_{k=0}^{i-1}.$ Using (\ref{eq:46})  we get
    \begin{align}
      \label{eq:47}
      Q(x,j) \le Q(x,j^1) + 2^{-i + 1} \|R_mx\|_\infty.
    \end{align}
    Pick $y$ satisfying (\ref{eq:43}). There is $r > m$ with $|y(r)|
    > \eps.$ We have
    \begin{align}
      \label{eq:48}
      (x(r) + y(r))^+ + (x(r) - y(r))^+ \ge |y(r)| - |x(r)| > \eps - \|R_mx\|_\infty.
    \end{align}
    Put $j^2 = (j_0, j_1, \ldots, j_{i-1}, r).$ Since $r > m \ge
    j_{i-1}$ we have $j^2 \in J.$ So
    \begin{align*}
      \|x \pm y\| \ge Q(x \pm y, j^2)
    = Q(x \pm , j^1) + 2^{-i}(x(r) \pm y(r))^+.
    \end{align*}
    Since $i \le j_{i-1}$ we get from (\ref{eq:46}), (\ref{eq:47}),
    and (\ref{eq:48})
    \begin{align*}
      (\|x + y\| + \|x - y\|)/2 & \ge Q(x, j^1) + 2^{-i-1}(\eps -
                                  \|R_m\|_\infty)\\
      & \ge Q(x,j) + 2^{-i-1}(\eps - 5\|R_m x\|_\infty) \\
      & \ge Q(x,j) + 2^{-m -1}(\eps - 5\|R_mx\|_\infty).
    \end{align*}
   This and (\ref{eq:44}), (\ref{eq:45}) imply (\ref{eq:42}).
  \end{proof}
   
   Let $x \in c_0, x_n \in \ell_\infty,$ and $\lim_n\vertiii{x \pm
     x_n} = \vertiii{x}.$ By convex arguments as in
   \cite[p.~45]{MR1211634} we have
   \begin{align}
     \label{eq:49}
     \lim_n \|x \pm x_n\| = \|x\|,
   \end{align}
    and
   \begin{align}
     \label{eq:50}
    \lim_n x_n(k) = 0, \hskip 2mm k = 0, 1, \ldots.  
   \end{align}
   We have to show that $\lim_n\|x_n\|_\infty = 0.$ Assume the
   contrary. Then we may assume that $\|x_n\|_\infty \ge 2\eps > 0$
   for all $n =1, 2, \ldots.$ From (\ref{eq:50}) it follows that there
   exists $m$ such that $\|P_m x_n\|_\infty < \eps$ for all $n$ which
   are sufficiently big. But, this contradicts (\ref{eq:49}) and Claim
   \ref{claim:1}. Hence $\vertiii{\cdot}$ is MLUR.

  Now, put $C = \{u \in c_0: \vertiii{u} \le 1\}.$ We will show that
  every non-void slice of $C$ has $\|\cdot\|_\infty$ diameter $\ge
  1/\sqrt{2}.$ To this end, let $f \in \ell_1, a \in \err, H = \{u \in
  c_0: f(u) > a\},$ and $S = H \cap C \not = \emptyset.$ 

  \begin{claim}\label{claim:2}
    For every $x \in H$ with $\vertiii{x} < 1,$ there is $y \in S$
    with $\vertiii{x-y} \ge \|x\|.$
  \end{claim}
  
  \begin{proof}
    Choose $\delta > 0$ such that
    \begin{align}
      \label{eq:51}
      (\|x\| + 2\delta)^2 + q^2(x) + \delta < 1, \hskip 2mm f(x) - \delta
      > a.
    \end{align}
    There exists a natural number $m$ such that
    \begin{align}
      \label{eq:52}
      \|R_m x\|_\infty < \delta,
    \end{align}
    \begin{align}
      \label{eq:53}
      0 < \|x\| - x(m) < \sqrt{2^m \delta},
    \end{align}
    and
    \begin{align}
      \label{eq:54}
      f(e_m) < \delta
    \end{align}
    where $e_m$ is the standard basis vector number $m$ in $c_0.$ Put
    $y = x - \|x\|e_m.$ Pick arbitrary $j \in J.$ We will show that
    \begin{align}
      \label{eq:55}
      Q(y,j) \le \|x\| + 2\delta.
    \end{align}
    If $m < j_0,$ then $y(j_k) = x(j_k)$ for $k = 0, 1, 2, \ldots.$ So
    \begin{align*}
      Q(y,j) = Q(x,j) \le \|x\|.
    \end{align*}
    If $m > j_0,$ we get $Q(y,j) \le Q(x,j) \le \|x\|$ since $y^+(m) =
    0.$ If $m =j_0$
     \begin{align*}
      Q(y,j) \le \|x\| - x(m) + \|R_mx\|_\infty.
    \end{align*}
    This and (\ref{eq:52}) imply (\ref{eq:55}). From (\ref{eq:53}) we
    get $q^2(y) < q^2(x) + \delta.$ From this, (\ref{eq:51}), and
    (\ref{eq:55}) we have $\vertiii{y} < 1.$ Since $\|x\| < 1$ we get
    from (\ref{eq:54}) that $y \in H.$ Thus $y \in S,$ and $\|x
    -y\|_\infty = \|x\|.$ Using (\ref{eq:41}) we get $\vertiii{x-y} \ge
    \|x - y\|_\infty = \|x\|,$ and the claim follows. 
  \end{proof}
  Since $x \in S$ can be chosen with $\vertiii{\cdot}$-norm
  arbitrarily close to $1$ we get from (\ref{eq:41}) that the diameter
  of $S \ge 1/\sqrt{2}.$  
\end{proof}

\section{Weak continuous perturbation of the norm}\label{sec:eqnorms}

It is natural to expect that weak continuous perturbations of the norm
would preserve points of continuity. We will show that similar
perturbations of the norm preserve strongly extreme points of the
corresponding ball.

\begin{thm}\label{thm:weak-npert}
  Let $X$ be a Banach space and let $B$ be the unit ball corresponding
  to an equivalent norm $\vertiii{\cdot}$ on $X.$ Assume
  that the restriction $\vertiii{\cdot}_{S_X}$ of $\vertiii{\cdot}$ to $S_X$
  is continuous at $x \in S_X$ 
\begin{enumerate}
  \item [a)] with respect to the weak topology. Then $x$ is a point of
    continuity for $B_X$ provided $x$ is a point of continuity for $B$.
   \item[b)] with respect to a topology $\sigma = \sigma(X,Y), Y
     \subset X^*$ with the property that for every bounded sequence in
     $X$ there exists a $\sigma$-Cauchy subsequence, and let $x$ be an
     extreme point of $B_X.$ Then $x$ is a strongly extreme point of
     $B_X$ provided it is a strongly extreme point for $B$.
  \end{enumerate}
\end{thm}

\begin{proof}
    From the assumption $\vertiii{x} = 1.$

    a). Pick $\eps > 0.$ Since $x$ is a point of continuity for
    $B$, there is a weakly open set $W$ containing $x$ such that
    $\mbox{diam}(W \cap B) < \eps.$ Since $\vertiii{\cdot}$ is
    uniformly norm continuous, we can find $\delta > 0$ such that
    $\mbox{diam} (W \cap (1 + 2\delta)B) < 2\eps.$ Since
    $\vertiii{\cdot}_{S_X}$ is weakly continuous at $x$, we can find
    a weakly open set $V$ containing $x$ such that $V \cap S_X \subset
    (1 + \delta)B.$ Since $\|\cdot\|$ is uniformly norm continuous, we
    can find $\eta \in (0,1)$ such that $V \cap (B_X \setminus \eta
    B_X) \subset (1 + 2\delta)B.$ Finally there is a weakly open set
    $U$ containing $x$ such that $U \cap \eta B_X = \emptyset.$ Hence
    $x \in U \cap V \cap W$ and $\mbox{diam}(U \cap V \cap W \cap
    B_X) < 2\eps.$
  
   b). Pick an arbitrary sequence $(u_n) \subset X$ with
     $\lim_n\|x \pm u_n\| = 1.$ In order to prove that $\lim\|u_n\|
     =0,$ it is enough to show that there exists a subsequence of
     $(u_n)$ converging to zero. To this end we assume without loss of
     generality that
     \begin{align*}
       \sigma-\lim_{m, n}(u_m - u_n) = 0.
     \end{align*}
     Taking into account that
     \begin{align*}
       \|x \pm \frac{u_m - u_n}{2}\| \le \frac{\|x \pm u_m\| + \|x -u_n\|}{2},
     \end{align*}
      we get
     \begin{align*}
      \lim_{m,n}\|x \pm \frac{u_m - u_n}{2}\| = 1. 
     \end{align*}
      Put $u_{m,n} = (u_m -u_n)/2,$ $\alpha_{m,n}^{\pm} = \|x \pm u_{m,n}\|
      -1.$ Clearly $\lim_{m,n} \alpha^{\pm}_{m,n} = 0$ and $\sigma
      -\lim_{m,n}(x \pm u_{m,n})/(1 + \alpha_{m,n}^\pm) = x,$ and $(x \pm
      u_{m,n})/(1 + \alpha_{m,n}^\pm) \in S_X.$ Since
      $\vertiii{\cdot}_{S_X}$ is $\sigma$-continuous at $x$, we get
      \begin{align*}
        \lim_{m,n}\vertiii{\frac{x\pm u_{m,n}}{1+\alpha^{\pm}_{m,n}}} = \vertiii{x} = 1.
      \end{align*}
     This implies $\lim_{m,n}\vertiii{x \pm u_{m,n}} = 1.$ Having in mind that
     $x$ is a strongly extreme point of $B,$ we get
     $\lim_{m,n}\vertiii{u_{m,n}} = 0.$ Hence $(u_n)$ is a norm Cauchy
     sequence. It follows that there is $u \in X$ with $\lim_n\|u_n -
     u\| = 0.$ Thus $\|x \pm u\| = 1,$ which implies that $u = 0$ as
     $x$ is an extreme point. Now $\lim_n\|u_n\|
     = 0$ which finishes the proof. 
\end{proof}

\begin{rem}
  The requirement that $x$ is an extreme point of $B_X$ is essential
  for part b). Indeed, in finite dimensional Banach spaces, all
  norms are weak($=$ norm) continuous. Clearly a point $x$ can be
  extreme for the ball $B$ and not extreme for $B_X$.
\end{rem}

We end with a result that follows directly by combining Theorems \ref{thm:LLT} and \ref{thm:weak-npert}.

\begin{cor} Let $\|\cdot\|_W$ be a weakly continuous semi norm on bounded sets in a Banach space $(X,\|\cdot\|)$ and let $|\cdot|$ be a lattice norm on $\mathbb{R}^2$. Let $\vertiii{x}=|(\|x\|,\|x\|_W)|$. Let $B$ be the unit ball of $(X,\vertiii{\cdot})$. Then the strongly extreme points of $B$ and $B_X$ coincide, and the denting points of $B$ and $B_X$ coincide.  
\end{cor}

\bibliographystyle{amsalpha}
\bibliography{sed_rev.bbl}

\newcommand{\etalchar}[1]{$^{#1}$}
\def\cprime{$'$} \def\cprime{$'$} \def\cprime{$'$}
\providecommand{\bysame}{\leavevmode\hbox to3em{\hrulefill}\thinspace}
\providecommand{\MR}{\relax\ifhmode\unskip\space\fi MR }
\providecommand{\MRhref}[2]{%
  \href{http://www.ams.org/mathscinet-getitem?mr=#1}{#2}
}
\providecommand{\href}[2]{#2}
\begin{thebibliography}{AHN{\etalchar{+}}16}

\bibitem[AHN{\etalchar{+}}16]{MR3499106}
T.~A.~Abrahamsen, P.~H{\'a}jek, O.~Nygaard, J.~Talponen, and
  S.~Troyanski, \emph{Diameter 2 properties and convexity}, Studia Math.
  \textbf{232} (2016), no.~3, 227--242. \MR{3499106}

\bibitem[DGZ93]{MR1211634}
R.~Deville, G.~Godefroy, and V.~Zizler, \emph{Smoothness and renormings in
  {B}anach spaces}, Pitman Monographs and Surveys in Pure and Applied
  Mathematics, vol.~64, Longman Scientific \& Technical, Harlow; copublished in
  the United States with John Wiley \& Sons, Inc., New York, 1993. \MR{1211634
  (94d:46012)}

\bibitem[GKS93]{GKS}
G.~Godefroy, N.~J.~Kalton, and P.~D.~Saphar, \emph{Unconditional ideals in
  {Banach} spaces}, Studia Math. \textbf{104} (1993), 13--59. \MR{1208038 (94k:46024)}

\bibitem[JL01]{MR1863688}
W.~B.~Johnson and J.~Lindenstrauss (eds.), \emph{Handbook of the geometry of
  {B}anach spaces. {V}ol. {I}}, North-Holland Publishing Co., Amsterdam, 2001.
  \MR{1863688}

\bibitem[KR82]{MR661446}
K.~Kunen and H.~Rosenthal, \emph{Martingale proofs of some geometrical
  results in {B}anach space theory}, Pacific J. Math. \textbf{100} (1982),
  no.~1, 153--175. \MR{661446}

\bibitem[LLT88]{LLT}
B.~L.~Lin, P.~K.~Lin, and S.~L.~Troyanski, \emph{Characterizations of denting
  points}, Proc. Amer. Math. Soc. \textbf{102} (1988), 526--528. \MR{0928972 (89e:46016)}

\end{thebibliography}

\end{document}